\documentclass{amsart}
\usepackage{amsmath,amsfonts,amsthm,amssymb,indentfirst,epic,url,graphics}
\usepackage{hyperref}

\newtheorem{theorem}{Theorem}
\newtheorem{lemma}[theorem]{Lemma}
\newtheorem{corollary}[theorem]{Corollary}
\newtheorem{proposition}[theorem]{Proposition}
\theoremstyle{definition}
\newtheorem{definition}[theorem]{Definition}
\newtheorem{question}[theorem]{Question}

\newcommand{\Z}{\mathbb{Z}}
\renewcommand{\r}{\mathrm}
\newcommand{\card}{\r{card}}

\begin{document}

\title%
{On Vaughan Pratt's crossword problem}
\thanks{%
After publication, any updates, errata, related references,
etc., found will be recorded at
\url{http://math.berkeley.edu/~gbergman/papers/}
and at \url{http://math.byu.edu/~pace/}.\\
\hspace*{1.05em}
Archived at \url{http://arXiv.org/abs/1504.07310}\,.\\
\hspace*{1.05em}
This work was partially supported by a grant from the
Simons Foundation (\#315828 to Pace Nielsen).
}


\author{George M. Bergman}
\address{University of California\\
Berkeley, CA 94720-3840, USA}
\email{gbergman@math.berkeley.edu}

\author{Pace P.\ Nielsen}
\address{Brigham Young University\\
Provo, UT 84602-1231, USA}
\email{pace@math.byu.edu}

\begin{abstract}
Vaughan Pratt has
introduced objects consisting of pairs $(A,\,W)$ where $A$
is a set and $W$ a set of subsets of $A,$ such that
\ (i)~$\!W$ contains $\emptyset$ and $A,$
(ii)~if $C$ is a subset of $A\times A$ such that
for every $a\in A,$ both $\{b\mid (a,b)\in C\}$
and $\{b\mid (b,a)\in C\}$ are members of $W$ (a ``crossword'' with
all ``rows'' and ``columns'' in $W),$
then $\{b\mid (b,b)\in C\}$ (the ``diagonal word'')
also belongs to $W,$ and
(iii)~for all distinct $a,b\in A,$ the set $W$ has an element which
contains $a$ but not $b.$
He has asked whether for every $A,$ the only such $W$ is
the set of {\em all} subsets of $A.$

We answer that question in the negative.
We also
obtain several positive results, in particular, a positive answer
to the above question if $W$ is closed under complementation.
We obtain partial results on whether there can exist counterexamples
to Pratt's question with $W$ countable.
\end{abstract}
\maketitle

\section{Definitions and conventions}\label{S.defs}

We begin by defining the type of structures we will be considering.
These are called ``$\!\mathbf{chu}_2\!$ comonoids'' by Vaughan Pratt;
we shall call them {\em Pratt comonoids}.
The category-theoretic background of Pratt's terminology
is not a prerequisite for reading
this note; we sketch that background in an appendix,~\S\ref{S.chu}.

\begin{definition}\label{D.chu_cm}
By a {\em Pratt comonoid} we shall mean a pair $(A,W),$ where
$A$ is a set, and $W$ a set of subsets of $A$ such that

\textup{(i)} $\emptyset$ and $A$ are members of $W,$ and

\textup{(ii)} whenever $C$ is a subset of $A\times A$ such that
for every $a\in A,$ both $\{b\mid (a,b)\in C\}$
and $\{b\mid (b,a)\in C\}$ are members of $W,$ we also have
$\{b\mid (b,b)\in C\}\in W.$

In this situation, we will call $A$ the {\em base-set} of
the Pratt comonoid $(A,W),$ and $W$ the {\em Pratt
comonoid structure} on $A.$

A set $W$ of subsets of a set $A$ will be called $\!T_1\!$
if it satisfies

\textup{(iii)} for all $a,b\in A$ with $a\neq b,$ the set $W$
has an element which contains $a$ but not $b.$

A Pratt comonoid $(A,W)$ will be called $\!T_1\!$ if $W$
is $\!T_1\!$ as a set of subsets of $A.$
It will be called {\em discrete} if
$W=2^A,$ the full power set of $A.$
\end{definition}

We shall generally identify subsets of $A$ or $A\times A$ with
$\!\{0,1\}\!$-valued functions on those sets.
In particular, we may call subsets of $A$ ``words''
on $A,$ and a subset $C\subseteq A\times A$ satisfying the
hypotheses of~(ii) a ``crossword'' over $W,$ since
its rows and columns are words lying in $W.$
We will use the notation $x\vee y$ and $x\wedge y$ for the
union and intersection (or from the $\{0,1\}$ point of view,
pointwise sup and pointwise inf) of words $x$ and $y,$
and likewise $\leq$ and $\geq$ for inclusion, and
$<$ and $>$ for strict inclusion between words.

For $C\subseteq A\times A$ and $a\in A,$ we shall follow
the matrix-theoretic convention of calling $\{b\mid (a,b)\in C\}$
the $\!a\!$-th row, and $\{b\mid (b,a)\in C\}$ the $\!a\!$-th column
(rather than the convention of the cartesian plane, where
the first member of an ordered pair is the horizontal
and the second the vertical coordinate).

This note was inspired by the following question, which will be
answered in the negative in~\S\ref{S.answer}.
\begin{equation}\begin{minipage}[c]{35pc}\label{d.VPq}
(V.\,Pratt \cite{puzzle}, \cite[pp.\,27--28]{chu_online}, \cite{chu})
\ Is every $\!T_1\!$ Pratt comonoid discrete?
\end{minipage}\end{equation}

Though~\eqref{d.VPq} concerns the $\!T_1\!$
case, many of the general results we prove will
concern arbitrary Pratt comonoids, with the $\!T_1\!$ condition
only brought in for the {\em coups de grace} of our main results.
Likewise, since our only known example of a $\!T_1\!$ Pratt comonoid
that is not discrete is quite complicated, examples showing the
obstructions to one or another approach
will in general be non-$\!T_1.$

We shall use set-theorists' notation for ordinals; in
particular, the set of natural numbers (nonnegative integers)
will be denoted $\omega.$
We shall write relative complements of sets as
$x-y=\{a\in x\mid a\notin y\}.$
When subsets of a given set $A$ are under consideration,
we shall often write $\neg x$ for the relative complement $A-x$
of a subset $x$ of $A.$

The next two sections mainly summarize known material.

\section{Some quick examples and immediate results}\label{S.quick}

The easiest examples of Pratt comonoids other than discrete
ones are gotten by taking a preorder $\preccurlyeq$ on a set $A,$
and defining $W$ to be the set of all {\em down-sets} of $A,$
that is, sets $x$ such that $a\preccurlyeq b\in x\implies a\in x$
\cite[Proposition~2.1]{comonoid}.
The reader can easily verify that
such pairs $(A,W)$ satisfy the definition.

For an example that deviates slightly from this form,
let $A$ consist of the set $\omega$ of natural numbers together
with one additional element $\infty,$ greater than every natural
number; and let
\begin{equation}\begin{minipage}[c]{35pc}\label{d.*w&infty}
$W\ =\ \{$down-sets of $A$ other than $\omega\,\}.$
\end{minipage}\end{equation}
In other words, $W$ consists of
those down-sets which, if they contain all natural
numbers, also contain $\infty.$
Since by the preceding paragraph, the set of {\em all} down-sets of $A$
yields a Pratt comonoid, to show that the $W$ we have just described
also determines one, we just have to show that any crossword over
$W$ whose diagonal contains all pairs $(n,n)$ $(n\in\omega)$ must
contain $(\infty,\infty).$
Now moving upward or to the left from the diagonal in such
a crossword $C$ (i.e., decreasing one or the other coordinate),
we see that every pair $(m,n)$ with $m,n\in\omega$
belongs to $C.$
Hence the word in $W$ given by each
natural-number-indexed row contains all natural numbers, hence,
by definition of $W,$ also contains $\infty.$
This says the
element of $W$ corresponding to the {\em column} indexed by $\infty$
contains all natural numbers, hence, again, contains $\infty;$ so
$(\infty,\infty)\in C,$ as required.

(The above example, essentially \cite[Proposition~2.4]{comonoid},
is an instance of the more general result
\cite[Proposition~2.2]{comonoid}, which says that what is called the
{\em Scott topology}
on a directed-complete partial order yields a Pratt comonoid.)

Here are some easy ways of getting new Pratt comonoids from old.

\begin{lemma}\label{L.new_fr_old}
\textup{(i)}  If $(A,W)$ is a Pratt comonoid, then
so is $(A,W^\neg),$ where $W^\neg=\{\neg w\mid w\in W\},$ the set of
complements in $A$ of members of $W.$
If $(A,W)$ is $\!T_1,\!$ then so is $(A,W^\neg).$

\textup{(ii)}  If $\{(A,W_i)\mid i\in I\}$ is a
\textup{(}finite or infinite\textup{)} set of Pratt comonoids
with the same base-set $A,$ then $(A,\,\bigcap_{i\in I} W_i)$
is a Pratt comonoid.

\textup{(iii)}  If $f: A\to A'$  is a set map and $(A,W)$ a
Pratt comonoid, and we let $W'=\{w\subseteq A'\mid f^{-1}(w)\in W\},$
then $(A',W')$ is a Pratt comonoid.

\textup{(iv)}  If $(A,W)$ is a Pratt comonoid, and $u\leq v$
are elements of $W,$ then the pair $(A_{u,v},W_{u,v}),$
where $A_{u,v} = v-u\subseteq A,$ and
$W_{u,v}=\{x-u\mid x\in W$ with $u\leq x\leq v\},$
is also a Pratt comonoid.
\end{lemma}

\begin{proof}[Sketch of proof]
We get (i)~by interchanging $\!0\!$'s and $\!1\!$'s in the conditions
defining a Pratt comonoid, and likewise in the $\!T_1\!$ condition.

Statement (ii)~holds because the condition for $(A,W)$ to be a
Pratt comonoid is a closure condition on $W,$ and for any
closure operator, an intersection of closed subsets is closed.

To see (iii), note that given any crossword $C'$ over $W',$ its
inverse image under $f\times f$ will be a crossword $C$ over $W,$
hence the diagonal thereof lies in $W,$ and that diagonal is the
inverse image of the diagonal of $C',$ which therefore lies in $W'.$

Finally, to see (iv), note that if $C\subseteq A_{u,v}\times A_{u,v}$
is a crossword over $W_{u,v},$ then $(u\times v)\vee(v\times u)\vee C$
will be a crossword over $W,$ hence its diagonal belongs
to $W,$ which translates to say
that the diagonal of $C$ belongs to $W_{u,v},$ as required.
\end{proof}

Part~(i) of the above lemma shows that when we prove a result
about Pratt comonoids, we can immediately get a dual statement
by applying that result to complements of words.
Likewise,~(iv) allows us to ``relativize'' any general result
about Pratt comonoids to yield a result about the set
of words $x\in W$ such that $u\leq x\leq v$
for given $u\leq v$ in~$W.$

Next, we note two easy ways of getting new words from old
within a given Pratt comonoid.

\begin{lemma}[{\cite[Proposition~2.5]{comonoid}}]\label{L.vee_wedge}
Let $(A,W)$ be a Pratt comonoid.
Then $W$ is closed under \textup{(}pairwise, hence finite\textup{)}
meets and joins.
That is, if $x,y\in W,$ then $x\wedge y$ and $x\vee y$ also
belong to $W.$
\end{lemma}

\begin{proof}
Given $x,y\in W,$ we find that $x\times y$ is a crossword
on $W$ (every row gives either the word $\emptyset$ or
the word $y;$ every column gives either $\emptyset$ or
$x),$ and its diagonal gives the word $x\wedge y,$
so $W$ is closed under intersections.

That it is also closed under unions follows by dualization,
i.e., applying the above result to the comonoid $(A,W^\neg)$
constructed as in Lemma~\ref{L.new_fr_old}(i).
Alternatively, one can verify directly that $(x\times A)\vee(A\times y)$
is a crossword on $W$ having $x\vee y$ as diagonal.
\end{proof}

Some observations on the above results:

Though we have seen that any intersection of Pratt comonoid
structures on a set $A$ is again a Pratt comonoid structure,
the same is not true of unions, even pairwise unions.
For instance, if $A=\{0,1,2\},$ and we let $W_\leq$ be the set of all
down-subsets of $A,$ and $W_\geq$ the set of up-subsets of $A,$ then
each of these is a Pratt comonoid structure, but their union
is not, since it contains both $\{0,1\}$ and $\{1,2\},$ but
not $\{0,1\}\cap\{1,2\}.$

Of course, since the condition of being a Pratt comonoid structure
on $A$ is a closure condition, there is a least such
structure containing the union of two given structures.
But the resulting
closure operation can expand the given union enormously.
For instance, suppose we let $A=\Z,$ the set of integers, again let
$W_\leq$ and $W_\geq$ be the systems of down-subsets and up-subsets
of $A,$ and let $W$ be the least Pratt comonoid structure
containing $W_\leq\cup W_\geq.$
By taking intersections, we see that $W$ contains all
singletons $\{i\}$ $(i\in\Z).$
Hence for any subset $x\subseteq\Z,$
the set $\{(i,i)\mid i\in x\}\subseteq\Z\times\Z$ is a
crossword over $W,$ since every row or column is either empty
or a singleton.
So $W$ consists of all subsets $x\subseteq \Z.$
Thus, closing under taking diagonals of crosswords has carried the
countable set $W_\leq\cup W_\geq$ to a set of continuum cardinality.

We saw in Lemma~\ref{L.vee_wedge} that each
Pratt comonoid structure on a set $A$ is closed under
pairwise unions and intersections.
However, such structures need not be closed under infinite unions
and intersections, as may be seen from the
example~\eqref{d.*w&infty} above, where $W$ contains
all the finite down-subsets of $\omega,$ but not their union,
$\omega$ itself.

Though we do not in general get all infinite unions and intersections,
everything we do get can be expressed in terms of such operations:

\begin{lemma}\label{L.diag_is_veewedge}
Let $A$ be any set, and $C:A\times A\to 2$ any map.
Then the subset $z\subseteq A$ corresponding to the diagonal of $C$
can be written as a \textup{(}possibly infinite\textup{)}
union of \textup{(}possibly infinite\textup{)} intersections
of subsets of $A$ corresponding to rows and columns of $C.$
\end{lemma}

\begin{proof}
For each $a\in A$ that occurs as a member of at least one row
or column of $C,$ let $x_a\subseteq A$ be the intersection of all
the rows and columns of $C$ that contain $a.$
Clearly, $a\in x_a.$
We claim that, in fact, $z=\bigvee_{a\in z} x_a.$
This will clearly imply the desired conclusion.

Indeed, for every $a\in z,$ the element $a$ lies in the above union,
namely, in the joinand indexed by $a.$
Conversely, if $b$ is in that union, this says it
lies in $x_a$ for some $a\in z.$
Looking at the $\!a\!$-th row of $C$ we conclude from the
definition of $x_a$ that since $C(a,a)=1,$ we have
$C(a,b)=1;$ and looking at the $\!b\!$-th column,
$C(a,b)=1$ similarly implies $C(b,b)=1.$
So $b\in z$ as required.
\end{proof}

The following easily verified observation (which does not
refer to diagonal words) will also be useful.

\begin{lemma}\label{L.2^*a}
If $A$ and $A'$ are sets, and $C$ a subset of $A\times A'$
which has precisely $\kappa$ distinct rows
\textup{(}resp.\ columns\textup{)} then it has at most $2^\kappa$
distinct columns \textup{(}resp.\ rows\textup{)}.

In particular, if $C$ has only finitely many distinct
rows \textup{(}columns\textup{)}, it has only finitely many
distinct columns \textup{(}rows\textup{)}.\qed
\end{lemma}

\section{The case where \texorpdfstring{$A$}{A} is countable}\label{S.countable}

We sketch below the proof of the known result that every $\!T_1\!$
Pratt comonoid $(A,W)$ with countable base-set $A$ is discrete.
First, a general observation.

\begin{lemma}\label{L.finite}
If $(A,W)$ is a $\!T_1\!$ Pratt comonoid and $A_0$
a finite subset of $A,$ then for any subset $y$ of $A_0$
there exists an $x\in W$ with $x\cap A_0 = y.$
\end{lemma}

\begin{proof}
For all $a,b\in A_0$ with $a\neq b,$ the $\!T_1\!$ condition
allows us to choose an $x_{a,b}\in W$ containing $a$ but not $b.$
By Lemma~\ref{L.vee_wedge}, the set
$x_a=\bigwedge_{b\in A_0-\{a\}} x_{a,b}$ is a member of $W$
containing $a$ but no other member of $A_0.$
We see that $x=\bigvee_{a\in y} x_a$ has the desired property.
\end{proof}

We can now get:

\begin{theorem}[V.\,Pratt {\cite[2nd exercise on p.\,28]{chu}}]\label{T.countable}
If $(A,W)$ is a $\!T_1\!$ Pratt comonoid, and $A$ is countable,
then $(A,W)$ is discrete.
\end{theorem}

\begin{proof}[Sketch of proof \textup{(after Mark G.\ Pleszkoch $\!=\!$ ``Mark Aujus'' \cite[Solution to puzzle 1.4]{puzzle})}]
Assume without loss of generality that $A=\omega.$
Take any $z\subseteq A,$ and assume inductively that for some $n$
we have found $x_0,\dots,x_{n-1},\,y_0,\dots,y_{n-1}\in W$
which are possible candidates for
the first $n$ rows and $n$ columns of a crossword having $z$ as
diagonal; i.e., such that the partial crossword formed by using the
$\!x_m\!$'s as its first $n$ rows and
the partial crossword formed by using the
$\!y_m\!$'s as its first $n$ columns agree
on their $n\times n$ intersection,
and the diagonal of that intersection yields the first
$n$ entries of $z.$

Now using the preceding lemma, with $A_0=\{0,\dots,n\},$
we can find $x_n,y_n\in W$ which extend our partial crossword;
i.e., such that the first $n$ entries
of $x_n$ are the entries of $y_0,\dots,y_{n-1}$ in the position
indexed by $n,$ while its next entry is the entry of $z$ needed in the
corresponding position on the diagonal;
and such that $y_n$ has the symmetric property.
This construction, continued recursively, leads to a full
crossword over $W$ with $z$ as diagonal.
\end{proof}

Now, on to new results.

\section{The complement-closed case}\label{S.complemented}

We shall prove in this section that if $(A,W)$ is a $\!T_1\!$ Pratt
comonoid, and $W$ is closed under complements, then
$(A,W)$ is discrete.
We begin with some observations on not
necessarily $\!T_1\!$ Pratt comonoids.

Though we have seen that for $(A,W)$ a Pratt comonoid,
$W$ need not be closed under infinite unions,
we claim that it is closed under unions
of families of subsets that are pairwise disjoint.
Here is a still more general statement.

\begin{lemma}\label{L.near_disjoint}
Suppose $(A,W)$ is a Pratt comonoid, and $x_i$ $(i\in I)$
are elements of $W$ such that for each $a\in A,$
only finitely many of the $x_i$ contain $a.$
Then $\bigvee_{i\in I} x_i\in W.$
\end{lemma}

\begin{proof}
Let $u=\bigvee_{i\in I}(x_i\times x_i)\subseteq A\times A.$
For each $a\in A,$ the $\!a\!$-th row (respectively,
the $\!a\!$-th column)
of $u$ is the union of finitely many of the $x_i,$ namely,
those that contain $a.$
Hence, since $W$ is closed under finite
unions, $u$ is a crossword over $W,$ hence its diagonal,
$\bigvee_{i\in I} x_i,$ indeed lies in $W.$
\end{proof}

\begin{corollary}\label{C.disjoint}
If $(A,W)$ is a Pratt comonoid, then $W$ is closed under
forming unions of disjoint families \textup{(}of arbitrary
cardinality\textup{)}.\qed
\end{corollary}

\begin{corollary}\label{C.complemented}
If $(A,W)$ is a Pratt comonoid such that $W$ is closed under
complementation, then $W$ is closed under
forming arbitrary unions.

Hence by duality, $W$ is also closed under
forming arbitrary intersections.
\end{corollary}

\begin{proof}
By Lemma~\ref{L.vee_wedge},
$W$ is closed under forming unions of finite families.
Let $\kappa$ be an infinite cardinal, assume inductively that
$W$ is closed under unions of families indexed
by sets of cardinality $<\kappa,$ and let $x_\beta$ $(\beta\in\kappa)$
be a $\!\kappa\!$-indexed family of members of $W.$

For each $\beta\in\kappa,$ our inductive hypothesis tells
us that $\bigvee_{\gamma<\beta} x_\gamma\in W,$ hence
since $W$ is closed under complementation,
the set $y_\beta=x_\beta-\bigvee_{\gamma<\beta} x_\gamma$
belongs to $W.$
The $y_\beta$ are easily seen to be pairwise disjoint and
to have union $\bigvee_{\beta\in\kappa} x_\beta.$
(Namely, each $a\in\bigvee_{\beta\in\kappa} x_\beta$
belongs to $y_\beta$ for $\beta$ the least ordinal
with $a\in x_\beta.)$
So by Corollary~\ref{C.disjoint}, that union belongs to $W.$
\end{proof}

We now get:

\begin{theorem}\label{T.complemented}
If $(A,W)$ is a $\!T_1\!$ Pratt comonoid such that $W$ is closed
under forming complements, then $(A,W)$ is discrete
\textup{(}i.e., $W$ is the set of all subsets of $A).$
\end{theorem}

\begin{proof}
Because $(A,W)$ is $\!T_1,\!$ for each $a\in A$ the intersection of
all members of $W$ that contain $a$ is $\{a\},$
so by the final assertion of
Corollary~\ref{C.complemented}, $W$ contains every singleton.
Since every subset of $A$ is a union of singletons, another application
of that corollary shows that $W$ contains every subset of $A.$
\end{proof}

\section{An interesting non-\texorpdfstring{$\!T_1\!$}{T1} Pratt comonoid structure on \texorpdfstring{$2^\omega$}{2\string^omega}}\label{S.2^*w}

In the first paragraph of \S\ref{S.quick}, we noted that easy
examples of Pratt comonoids $(A,W)$ in which $W$ was not the
set of all subsets of $A$ could be gotten by starting
with any partially ordered set $A,$ and letting $W$ be the
set of all its down-sets; and
in the next paragraph, we noted a case where such an $A$
admitted a slightly smaller comonoid structure $W,$ determined
by a sort of ``continuity'' condition.
Below, we construct another example of a sub-comonoid
of the Pratt comonoid arising from a partially ordered set,
which differs from it much more strikingly:  The partially
ordered set $A$ we start with will be $2^\omega,$ ordered
by inclusion; the partially
ordered set of {\em all} its down-sets can be shown to
have cardinality $2^{2^{\aleph_0}},$ but our $W$ will be countable.
This example will be a key ingredient in our construction,
in the next section, of a non-discrete $\!T_1\!$ example.

Actually, we will construct the Pratt comonoid of this
section as a sub-comonoid
of the comonoid of all {\em up-sets} of $2^\omega,$ i.e., families of
subsets closed under enlargement.
The up-sets of any partially ordered set $A$ form a
Pratt comonoid structure on $A$ for the same
reason that the down-sets do,
and since the partially ordered set $2^\omega$ is isomorphic
to its opposite, the two comonoids are isomorphic
(cf.\ Lemma~\ref{L.new_fr_old}(i)).
We will use up-sets because it will be conceptually simpler to
take for the building blocks of our construction the
up-sets $e_n$ $(n\in\omega)$ consisting of all subsets of $\omega$
that contain $n,$ rather than the down-sets given by their complements.

Recall that for every set $A,$ the set $2^A$ of all subsets of $A,$
regarded as a direct product of copies of the discrete topological
space $2,$ is a compact Hausdorff space (compact by
Tychonoff's Theorem).
A subbasis for its open sets
is given by the sets $\{x\in 2^A\mid a\in x\}$ for $a\in A,$
and their complements, $\{x\in 2^A\mid a\notin x\}.$
We shall call this topology the {\em natural topology} on $2^A.$

In particular, both $2^\omega$ and $2^{2^\omega}$ have such topologies;
note that the definition of the topology on the latter set uses
only the set-structure of the former, and ignores its topology.
(But a relation between the topologies of these two sets
will be key to the proof of the final result of this section.)

We start with some observations on a general partially
ordered set $A.$
Recall that if $(A,\preccurlyeq)$ is a partially ordered set
and $a\in A,$ then ${\uparrow}(a)=\{b\in A\mid b\succcurlyeq a\}$
is called the {\em principal up-set} determined by $a.$
Likewise ${\downarrow}(a)=\{b\in A\mid b\preccurlyeq a\}$
is called the {\em principal down-set} determined by $a.$
(One normally calls these the principal up-set and principal down-set
{\em generated} by $a,$ but we will use ``determined'' to
avoid confusion with comonoids generated by sets of subsets.)

\begin{lemma}\label{L.isolated}
Let $A$ be a set given with a partial ordering $\preccurlyeq,$
and let $U_\preccurlyeq(A)\subseteq 2^A$ denote the set of up-sets
of $A,$ with the topology induced by the natural topology on $2^A.$
Then the following conditions on an element $x\in U_\preccurlyeq(A)$
are equivalent.

\textup{(i)}  $x$ is an isolated point of the topological
space $U_\preccurlyeq(A);$
i.e., $x$ is not in the closure of $U_\preccurlyeq(A)-\{x\}.$

\textup{(ii)}  $x$ is both the union of a finite \textup{(}possibly
empty\textup{)} family of principal up-sets, and
the intersection of a finite \textup{(}possibly
empty\textup{)} family of complements of principal down-sets.
\end{lemma}

\begin{proof}
To prove (ii)$\implies$(i), note that
if $x$ is a union ${\uparrow}(a_0)\vee\dots\vee{\uparrow}(a_{m-1}),$
then it is the smallest (under inclusion) element of
$U_\preccurlyeq(A)$ containing all of $a_0,\dots,a_{m-1}.$
Likewise, if it is an intersection
$\neg{\downarrow}(b_0)\wedge\dots\wedge\linebreak[1]
\neg{\downarrow}(b_{n-1}),$ then it is the largest element of
$U_\preccurlyeq(A)$ not containing any of $b_0,\dots,b_{n-1}.$
These conditions together make it the {\em unique} element of
$U_\preccurlyeq(A)$ containing each of the $a_i$ and none of the $b_j.$
Now the property of containing or not containing a specified
element of $A$ defines an open subset of $2^A,$ hence
of $U_\preccurlyeq(A).$
Thus, intersecting the $m+n$ open sets arising
from the above description, we get an open subset of $U_\preccurlyeq(A)$
having $x$ as its only point; so $x$ is isolated.

Conversely, since the subsets of $U_\preccurlyeq(A)$ defined by the
conditions of containing or not containing a given element of $A$ form
a subbasis of its open sets, if $x$ is isolated it must be
the unique point in a finite intersection of such sets;
i.e., the unique up-set that contains all
members of a finite family of points $a_0,\dots,a_{m-1}$
and no members of another finite family $b_0,\dots,b_{n-1}.$
It is easy to see that there is a least up-set containing
$a_0,\dots,a_{m-1}$ (namely,
${\uparrow}(a_0)\vee\dots\vee{\uparrow}(a_{m-1}))$
and a greatest containing none of $b_0,\dots,b_{n-1}$ (namely,
$\neg{\downarrow}(b_0)\wedge\dots\wedge\neg{\downarrow}(b_{n-1})).$
Given that the families of up-sets defined by these two properties
have nonempty intersection, if the least member of one family
and the greatest member of the other did not coincide,
then the intersection of the two families would not be a singleton.
So they do coincide, giving a description of $x$ as in~(ii).
\end{proof}

(M.\,Ern\'{e} has kindly pointed out to us that the above result
can be deduced from the general theory of continuous
lattices \cite{GHKLMS1980}, \cite{GHKLMS2003}.
Namely, in any lattice which is {\em bicontinuous}
in the sense of \cite[Chapter~VII]{GHKLMS2003}, the points
that are isolated in the bi-Scott topology
are those that are both {\em isolated from above} and
{\em isolated from below} in the sense of \cite[Chapter~I]{GHKLMS2003}.
A subclass of the bicontinuous lattices are the
{\em superalgebraic} lattices \cite{abc} \cite{ME}, which are,
up to isomorphism, the up-set lattices of partially
ordered sets $A;$ and in these, the elements isolated
from below are the finitely generated upsets,
those isolated from above are the complements
of finitely generated down-sets, and the bi-Scott topology
agrees with the topology induced by the natural topology on $2^A,$
yielding the statement of the lemma.)

We now apply the above lemma to the case where $A$ is
the set $2^\omega,$ partially ordered by inclusion.
(We could allow any set in place of $\omega,$ but we shall
see in \S\ref{S.*P_Pi} that this example can be generalized in other
ways; so we will just consider here the case we are about to use.)
Conditions~(i) and~(ii) below are as in the lemma;~(iii) is what is new.

\begin{corollary}\label{C.isolated}
Let $A=2^\omega,$ partially ordered by inclusion, $\subseteq,$
and $U_\subseteq(A)\subseteq 2^A$ its set of up-sets.
For each natural number $n,$ let $e_n\in U_\subseteq(A)$
denote the set of all subsets of $\omega$ containing $n.$
Then the following conditions on an element $x\in U_\subseteq(A)$
are equivalent.

\textup{(i)} $x$ is an isolated point of $U_\subseteq(A)$
under the natural topology.

\textup{(ii)}  $x$ is both the union of a finite \textup{(}possibly
empty\textup{)} family of principal up-sets, and
the intersection of a finite \textup{(}possibly
empty\textup{)} family of complements of principal down-sets of $A.$

\textup{(iii)} $x$ lies in the lattice generated
by $\{e_n\mid n\in\omega\}\cup\{\emptyset,A\};$ i.e., the closure of
that set under pairwise unions and intersections.
\end{corollary}

\begin{proof}
By Lemma~\ref{L.isolated}, (i)$\iff$(ii), so it will suffice
to show that (iii)$\implies$(ii), and that
(i)$\wedge$(ii)$\implies$(iii).

Assume~(iii).
If $x=\emptyset,$ it is both the union of the empty
family of principal up-sets and the intersection of
a $\!1\!$-element family of complements of principal
down-sets, namely $\{\neg{\downarrow}(A)\},$ so it satisfies~(ii).
The case $x=A$ is seen similarly.

If $x$ is neither $\emptyset$ nor $A,$ it can be
written as a lattice-theoretic expression in the $e_n,$
and using distributivity, we can express it both as a
finite join of finite meets of these elements, and as
a finite meet of finite joins thereof.
Using the former expression, we note that each finite meet
$e_{n_0} \wedge\dots\wedge e_{n_{i-1}}$ is
the principal up-set determined by $\{n_0,\dots,n_{i-1}\},$
so we have the first condition of~(ii).

On the other hand, when we express $x$ as a finite meet
of finite joins of the $e_n,$ each of those finite joins
can be looked at as the complement of a finite meet
of complements of the $e_n;$ and we see that such a meet
$\neg e_{n_0}\wedge\dots\wedge\neg e_{n_{i-1}}$ is
a principal down-set, the set of elements of $A$ that
are $\leq\neg\{n_0,\dots,n_{i-1}\}.$
So $x$ is a finite meet of complements
of principal down-sets, giving the second condition of~(ii).

Conversely, assume (i)$\wedge$(ii).
By~(ii), $x$ can be written as a finite join
of principal up-sets,
$x={\uparrow}(s_0)\vee\dots\vee{\uparrow}(s_{n-1})$
$(s_0,\dots,s_{n-1}\subseteq\omega).$
Without loss of generality we may assume that none of the $s_i$
contains any of the others.
Suppose one of them, $s_i,$ were infinite.
Then ${\uparrow}(s_i)$ is the intersection of the
downward-directed set of up-sets ${\uparrow}(s)$ as $s$ ranges over the
finite subsets of $s_i,$ and we see that it will be the
limit of those up-sets under the topology on $U_\subseteq(A).$
Now holding the other $s_j$ in our expression for $x$ fixed, and
letting $s\to s_i$ as above, we get a family of up-sets approaching $x.$
Moreover, if one of these up-sets by which we are
approaching $x$ coincided with $x,$ say the one
constructed from a finite subset $s\subseteq s_i,$
then $x$ would have $s$ as a member, which it does not,
since none of the other $\!s_j\!$'s is contained in $s_i.$
Thus $x$ is a limit of points distinct from $x,$ contradicting~(i).
So all $s_i$ are finite, hence each ${\uparrow}(s_i)$ is a
finite (possibly empty) meet of the $e_n,$ so $x$ is a finite
(possibly empty) join of such finite meets, proving~(iii).
\end{proof}

We can now prove:

\begin{theorem}\label{T.2^*w}
Let $A=2^\omega;$ for each natural number $n$
let $e_n\subseteq A$ be the set of subsets of $\omega$ containing $n,$
and let $W\subseteq 2^A$ be the closure of
$\{e_n\mid n\in\omega\}\cup\{\emptyset,A\}$ under pairwise unions and
intersections.
Then $(A,W)$ is a Pratt comonoid.
\end{theorem}

\begin{proof}
We first note that every $x\in W,$ regarded as a
function $A\to 2,$ is continuous with respect to the
product topology on $A=2^\omega$ and the discrete topology on $2.$
Indeed, the generators $e_n$ and their
complements are the characteristic functions of
the open-closed generating sets for the topology on $A,$
hence are continuous, as are the constant functions $\emptyset$ and $A;$
and the general member of $W$ is obtained from the $e_n,$
$\emptyset,$ and $A$ using the
operations $\wedge$ and $\vee$ on $2,$ which are necessarily
continuous in the discrete topology on $2.$

Thus, if $C:A\times A\to 2$ is any crossword whose {\em rows}
are given by elements of $W,$ then each row represents a
continuous map $A\to 2.$
Hence the whole crossword, thought of as a map taking each element
of $A$ to the column it indexes, is a continuous function $A\to 2^A.$

Suppose now that $C$ as above has infinitely many distinct columns.
Let $B$ be an infinite subset of $A$ whose
members index distinct columns of $C.$
Since $A$ is compact, we can find some limit-point
$b\in A$ of $B.$
By continuity, the column of $C$ indexed by $b$ is
a limit-point of columns indexed by the elements of $B.$
Hence, since elements of $W$ are isolated points, the
$\!b\!$-th column of $C$ cannot be a member of $W.$

It follows that if all rows {\em and} columns of $C$ belong
to $W,$ then $C$ can have only finitely many distinct columns,
and thus by Lemma~\ref{L.2^*a}, also only finitely many distinct rows.
Hence we can apply Lemma~\ref{L.diag_is_veewedge} (with its
qualifier ``{\em possibly infinite}\/'' irrelevant because of the above
finiteness results), and the fact that $W$ is closed under
pairwise unions and intersections, to conclude that the diagonal of $C$
is a member of $W;$ proving that $(A,W)$ is a Pratt comonoid.
\end{proof}

We will have further use for the technique of the last sentence
of the above proof; so let us record what that argument gives us.

\begin{corollary}[to Lemmas~\ref{L.vee_wedge} and~\ref{L.diag_is_veewedge}]\label{C.fin_many}
Suppose $A$ is a set and $W$ a set of subsets of $A$ which
is closed under pairwise unions and intersections,
and contains $\emptyset$ and $A.$
Suppose, moreover, that no crossword $C: A\times A\to 2$ with
all rows and columns in $W$ has infinitely many distinct rows
\textup{(}equivalently, by Lemma~\ref{L.2^*a},
has infinitely many distinct columns\textup{)}.
Then $(A,W)$ is a Pratt comonoid.\qed
\end{corollary}

\section{A non-discrete \texorpdfstring{$\!T_1\!$}{T1} Pratt comonoid}\label{S.answer}

We are now ready to construct a Pratt comonoid $(A,W)$
that answers the question~\eqref{d.VPq} in the negative.

Intuitively, the idea will be to make our base-set $A$ the
union of many ``islands'', and take
$W$ to be generated by an uncountable family of
subsets $w_{n,\gamma}$ of $A,$ such that each generator, when
restricted to certain of the islands,
looks like one of the generators in the example of
Theorem~\ref{T.2^*w} above, while everywhere else, it
looks like one of a countable family of generators $u_n$ of a discrete
Pratt comonoid structure on $A.$

The fact that the $w_{n,\gamma}$ look ``in most places'' like the $u_n$
will make our structure $\!T_1.\!$
On the other hand, the system of ``islands'' will be
set up so that given any {\em countable} family of the $w_{n,\gamma},$
there is some island on which that
family acts precisely like our generating set for the
construction of Theorem~\ref{T.2^*w}.
We will use this property to show that, as in that theorem, no
crossword formed from lattice expressions in our generators
can have infinitely many distinct rows or columns; whence those
lattice expressions will in fact form a Pratt comonoid structure on $A,$
which we shall see is $\!T_1\!$ but not discrete.

We begin with an easy general observation.

\begin{lemma}\label{L.ctb_T1}
If $A$ is a set of continuum cardinality, then there is a countable
$\!T_1\!$ family $\{u_0,u_1,\dots\}$ of subsets of $A.$
\end{lemma}

\begin{proof}
It suffices to construct such a family for $A=2^\omega.$
To do this, let us define each even-indexed
set $u_{2n}$ to be the set $e_n$ of subsets of $\omega$ which
contain $n,$ and each odd-indexed set $u_{2n+1}$ to be
the set $\neg e_n$ of subsets of $\omega$ which do not contain $n.$
The $\!T_1\!$ property is immediate.
\end{proof}

Recalling that $\omega_1$ denotes the first uncountable ordinal,
we now define the base-set $A$ of our example:
\begin{equation}\begin{minipage}[c]{35pc}\label{d.ceg_A}
$A\ =\ A'\times A'',$ where

\hspace*{1em}$A'\ =$ the set of all one-to-one maps
$\omega\to\omega\times\omega_1,$

\hspace*{1em}$A''\ =\ 2^\omega.$
\end{minipage}\end{equation}

The typical element of $A$ will be written $(a',a''),$
with $a'\in A',$ $a''\in A''.$
The ``islands'' referred to in the above sketch
will be the sets $\{a'\}\times A''.$

Observe that $A$ has the cardinality of the continuum.
Indeed, since $A''$ has that cardinality, it suffices to show
that $A'$ has at most that cardinality.
Now $A'$ is contained in the set of {\em all} maps
$\omega\to\omega\times\omega_1,$ and the cardinality of that
set is bounded above
by the result of replacing $\omega\times\omega_1$ in that
description by $2^\omega,$ i.e., by the cardinality of
the set of maps $\omega\to 2^\omega,$ which
is $(2^{\aleph_0})^{\aleph_0}=
2^{\aleph_0\cdot\aleph_0}=2^{\aleph_0},$ as required.

Hence by Lemma~\ref{L.ctb_T1}, we can
\begin{equation}\begin{minipage}[c]{35pc}\label{d.u_n}
let $\{u_n\mid n\in\omega\}$ be a countable $\!T_1\!$ family of subsets
of $A.$
\end{minipage}\end{equation}

We now define the family of subsets of $A$ (which we will
express as $\!\{0,1\}\!$-valued functions) that we will take as the
generators of our Pratt comonoid.
Namely, for each $(n,\gamma)\in\omega\times\omega_1,$
we let $w_{n,\gamma}:A\to 2$ be defined by
\begin{equation}\begin{minipage}[c]{35pc}\label{d.wn*m}
$w_{n,\gamma}(a',a'')\ = \left\{ \begin{array}{ll}
u_n(a',a'') & \mbox{if $(n,\gamma)$ is not among the
elements $a'(i)$ for $i\in\omega$ (cf.~\eqref{d.ceg_A}),}\\[.2em]
a''(i) & \mbox{if $i\in\omega$ satisfies $a'(i)=(n,\gamma).$}
\end{array}\right.$
\end{minipage}\end{equation}

Thus, for each ``island'' $\{a'\}\times A'',$
the coordinate $a'$ specifies a countable
sequence of ordered pairs $(n,\gamma)$ such that the corresponding
generators, $w_{n,\gamma},$ act on the $\!A''\!$-components
of members of that island
like the generating functions $e_i$ of the example of
Theorem~\ref{T.2^*w}; namely, if $(n,\gamma)=a'(i),$
then $w_{n,\gamma}$ selects the $\!i\!$-th coordinate of $a''.$

Now let
\begin{equation}\begin{minipage}[c]{29pc}\label{d.W_frm_wn*m}
\hspace*{-2.48pc}$W=$ the lattice of subsets
of $A$ generated by
$\{w_{n,\gamma}\mid n\in\omega,\,\gamma\in\omega_1\}\cup\{\emptyset,A\}$
under pairwise unions and intersections.
\end{minipage}\end{equation}

Most of our work will go into showing that $(A,W)$ is a Pratt comonoid.

Turning back to the definition~\eqref{d.ceg_A} of $A,$
observe that each $a'\in A'$ has countable image; hence the set
of second coordinates of elements of that image will be
bounded within $\omega_1.$
Hence defining, for each $\beta<\omega_1,$
\begin{equation}\begin{minipage}[c]{35pc}\label{d.A*l}
$A_\beta\ =\ \{(a',a'')\in A\mid$ the second coordinates of all
elements of the image of $a'$ are $<\beta\},$
\end{minipage}\end{equation}
we have
\begin{equation}\begin{minipage}[c]{35pc}\label{d.A=bigcup}
$A$ is the union of the chain of subsets
$A_\beta$ $(\beta\in\omega_1),$
\end{minipage}\end{equation}
and we see from the first line of~\eqref{d.wn*m} that
\begin{equation}\begin{minipage}[c]{35pc}\label{d.like_un}
for $\beta\in\omega_1,$ every $w_{n,\gamma}$ with $\gamma\geq\beta$
acts on $A_\beta$ by $u_n.$
\end{minipage}\end{equation}

We can now deduce that
\begin{equation}\begin{minipage}[c]{35pc}\label{d.wn*mT1}
$W$ is $\!T_1\!$ on $A.$
\end{minipage}\end{equation}
Namely, given $a_1,\ a_2\in A,$ the $\!T_1\!$ property
of the $u_n$ lets us choose an $n$ such that $u_n(a_1)=1,$ $u_n(a_2)=0,$
and~\eqref{d.A=bigcup} allows us to choose a $\beta$ such
that $a_1,\,a_2\in A_\beta;$
so by~\eqref{d.like_un}, $w_{n,\beta}(a_1)=u_n(a_1)=1,$
$w_{n,\beta}(a_2)=u_n(a_2)=0,$ as required.

Let us show next that
\begin{equation}\begin{minipage}[c]{35pc}\label{d.ctbl_on_A*l}
for every $\beta\in\omega_1,$ the restrictions to
$A_\beta\subseteq A$ of the
elements of $W\subseteq 2^A$ are countable in number.
\end{minipage}\end{equation}
By~\eqref{d.W_frm_wn*m} it suffices to show
that the restrictions to $A_\beta$ of the
generating elements $w_{n,\gamma}$ are countable in number.
By~\eqref{d.like_un}, those $w_{n,\gamma}$ with $\gamma\geq\beta$
have restrictions given by the countably many elements $u_n.$
By definition of $\omega_1,$
there are only countably many $\gamma<\beta,$ and hence only countably
many $w_{n,\gamma}$ with such $\gamma,$ completing the proof
of~\eqref{d.ctbl_on_A*l}.

We need, next, a combinatorial lemma (which we will apply
to occurrences of the $w_{n,\gamma}$ appearing as arguments in a
$\!j\!$-variable lattice term).

\begin{lemma}\label{L.all_or_1}
Let $X$ be a set, $j$ a positive integer, and $S$ an infinite
set of ordered $\!j\!$-tuples of elements of $X,$ such
that in each member of $S,$ the $j$ entries are distinct.

Then there exist an infinite subset $S'\subseteq S$ and an
$i<j$ such that, after we apply some permutation of the
$j$ coordinates to all our $\!j\!$-tuples, all members of $S'$
begin with the same initial $\!i\!$-element string, while
every element of $X$ that occurs in one of the last $j-i$ positions of
an element of $S'$ occurs in no other member of $S'.$

Thus, cutting $S'$ down to a countable set if it was uncountable,
and writing $k=j-i>0,$ we can find distinct
elements $x_\ell\in X$ $(\ell\in\omega)$ such that $S'$
consists of the $\!j\!$-tuples
\begin{equation}\begin{minipage}[c]{35pc}\label{d.j-tuples}
$(x_0,\dots,x_{i-1},\,x_{i+hk},\dots,x_{i+(h+1)k-1})$\quad
for $h\in\omega.$
\end{minipage}\end{equation}
\end{lemma}

\begin{proof}
Let $i<k$ be the largest integer such that $S$ contains infinitely
many elements which agree in some
common $\!i\!$-tuple of their coordinates;
let us perform a permutation of indices that makes $0,\dots,i{-}1$
such a set of coordinates, and let us choose $x_0,\dots,x_{i-1}$
which appear in that
order as the first $i$ coordinates of infinitely many
elements of $S.$
Let $S_0\subseteq S$ be the infinite set of those elements of $S$
beginning with the string $x_0,\dots,x_{i-1}.$
Note that the maximality of $i$ implies that no element of $X$
other than $x_0,\dots,x_{i-1}$ occurs in infinitely many
members of $S_0.$

Now let $k=j-i,$ and choose any $x_i,\dots,x_{i+k-1}$ such that
$(x_0,\dots,x_{i-1},\,x_i,\dots,x_{i+k-1})\in S_0.$
As just noted, each of $x_i,\dots,x_{i+k-1}$ appears
as an entry in only finitely many members of $S_0,$ so
if we let $S_1$ be the set of elements of $S_0$ in which
none of them appear, this will still be infinite, and we can
pick $x_{i+k},\dots,x_{i+2k-1}$ such that
$(x_0,\dots,x_{i-1},\,x_{i+k},\dots,x_{i+2k-1})\in S_1.$
Letting $S_2$ be the infinite set of elements of $S_1$
involving none of $x_{i+k},\dots,x_{i+2k-1},$
we can similarly pick $x_{i+2k},\dots,x_{i+3k-1}$
to get an element of $S_2;$ and so forth.
Thus we get an infinite family $S'$ of the desired form.

(We shall only need a countably infinite $S';$ but if we wished,
we could get $S'$ to have any regular cardinality $\kappa\leq\card(S),$
by replacing references to finite and infinite subsets of $S$
in the above proof with subsets of cardinalities $<\kappa$ and
$\geq\kappa.)$
\end{proof}

Using this lemma, let us prove:

\begin{lemma}\label{L.unctbl}
For $A$ and $W$ defined as in \eqref{d.ceg_A}-\eqref{d.W_frm_wn*m},
every uncountable subset $W_0\subseteq W$ contains a
countable subset $W_1$ such that
the set of distinct functions $W_1\to 2$ obtained by evaluation
at different elements of $A$ has continuum cardinality.
\end{lemma}

\begin{proof}
Each member of $W$ can be written as a lattice expression
$b(w_{n_0,\gamma_0},\dots,w_{n_{j-1},\gamma_{j-1}})$
of some finite length $j,$ whose arguments
$w_{n_0,\gamma_0},\dots,w_{n_{j-1},\gamma_{j-1}}$ are distinct, and
which depends nontrivially on all $j$ of its arguments
(i.e., such that inserting all combinations
of $\!0\!$'s and $\!1\!$'s in those $j$ positions in $b,$
the resulting function depends on each of its variables).
Note that, ignoring the choice of variables
$w_{n_0,\gamma_0},\dots,w_{n_{j-1},\gamma_{j-1}},$
there are only countably many distinct finite lattice
terms; hence the uncountability of $W_0$ implies that there is
some lattice term $b,$ say in $j$ variables,
from which one can get infinitely many of the
members of $W_0$ by inserting appropriate elements $w_{n,\gamma}$
as its arguments.
Let us fix such a $b.$

We can now apply Lemma~\ref{L.all_or_1} with $X=\{w_{n,\gamma}\},$
and $S$ the set of $\!j\!$-tuples of elements of $X$
which, when used as the argument-string of $b,$ give an
element of $W_0.$
That lemma gives us a sequence of distinct
elements $w_{n_\ell,\gamma_\ell}$ $(\ell\in\omega)$ such that
\begin{equation}\begin{minipage}[c]{35pc}\label{d.all_h}
for each $h\in\omega,$\quad
$b(w_{n_0,\gamma_0},\dots,w_{n_{i-1},\gamma_{i-1}},
w_{n_{i+hk},\gamma_{i+hk}},\dots,
w_{n_{i+(h+1)k-1},\gamma_{i+(h+1)k-1}})\in W_0.$
\end{minipage}\end{equation}
In particular, the set
\begin{equation}\begin{minipage}[c]{35pc}\label{d.W_1}
$W_1\ =\ \{\,b(w_{n_0,\gamma_0},\dots,w_{n_{i-1},
\gamma_{i-1}},w_{n_{i+hk},\gamma_{i+hk}},\dots,
w_{n_{i+(h+1)k-1},\gamma_{i+(h+1)k-1}})\mid h\in\omega\}$
\end{minipage}\end{equation}
will be a countably infinite subset of $W_0.$
Having so chosen the $n_\ell$ and $\gamma_\ell,$ let us encode them as
an element $a'\in A',$ setting
\begin{equation}\begin{minipage}[c]{35pc}\label{d.a'=}
$a'(\ell)\ =\ (n_\ell,\gamma_\ell)$ for $\ell\in\omega.$
\end{minipage}\end{equation}

To complete the proof of the lemma, we shall construct for this $a'$
a continuum-sized family of points of $\{a'\}\times 2^\omega$
such that restriction
to distinct points of that family induces distinct
$\!\{0,1\}\!$-valued functions on the countable set $W_1.$

To do this, recall that our lattice-expression $b$
depends on all $j=i+k$ of its arguments.
Combining this with the fact that lattice operations
are isotone (order-respecting), we see that for some choice of
values $c_0,\dots,c_{i-1}\in\{0,1\},$ we will have
\begin{equation}\begin{minipage}[c]{35pc}\label{d.b=0,1}
$b(c_0,\dots,c_{i-1},0,\dots,0)=0,$\quad
$b(c_0,\dots,c_{i-1},1,\dots,1)=1.$
\end{minipage}\end{equation}

Let us fix such $c_0,\dots,c_{i-1}\in\{0,1\}.$
Within $\{a'\}\times 2^\omega,$
let us now choose the $\!2^\omega\!$-tuple consisting of all elements
whose $\!A''\!$-coordinates have the values $c_0,\dots,c_{i-1}$
at $0,\dots,i-1,$ then have a common
value, $0$ or $1,$ at $i,\dots,i+k-1,$
likewise a common value, $0$ or $1,$
at $i+k,\dots,i+2k-1,$ and, generally, for each $h\in\omega,$
a common value, $0$ or $1,$ at $i+hk,\dots,i+(h{+}1)k-1.$
We see from~\eqref{d.wn*m} and~\eqref{d.b=0,1}
that on the $\!\omega\!$-tuple of elements
comprising $W_1,$ this family of
points of $A$ will induce all $2^\omega$ possible
$\!\omega\!$-tuples of values in $\{0,1\},$
yielding the assertion of the lemma.
\end{proof}

Contrasting~\eqref{d.ctbl_on_A*l} and Lemma~\ref{L.unctbl}, we get:

\begin{corollary}\label{C.unctbl_C}
Let $A$ and $W$ be defined as in \eqref{d.ceg_A}-\eqref{d.W_frm_wn*m},
and $C$ be any function $A\times A\to\{0,1\}.$

Then if $C$ has uncountably many rows which give distinct values in $W,$
it has at least one column which is not in $W.$
Likewise, if $C$ has uncountably many columns giving distinct values in
$W,$ it has at least one row not in~$W.$
\end{corollary}

\begin{proof}
If $C$ has rows with uncountably many distinct values in $W,$
let $W_0\subseteq W$ be the set of values of these rows.
Choose $W_1\subseteq W_0$ as in Lemma~\ref{L.unctbl}, and choose
$\beta\in\omega_1$ such that all of the elements of $W_1$ occur
as rows indexed by members of $A_\beta.$
(This can be done because $W_1$ is countable, while
$\omega_1$ has uncountable cofinality.)
Then by choice of $W_1$ and $\beta,$
the restrictions of the columns of $C$
to $A_\beta$ give continuum many distinct functions $A_\beta\to 2.$
But by~\eqref{d.ctbl_on_A*l}, only countably many of these
restrictions can arise from members of $W;$ so not all columns
of $C$ belong to $W.$
The statement with rows and columns interchanged is seen in
the same way.
\end{proof}

Now assume that $C$ is a crossword over $W.$
The above corollary shows that $C$ can
have at most countably many distinct rows
and at most countably many distinct columns.
Hence the rows and columns of $C$
are expressible as lattice-theoretic expressions
in countably many of the $w_{n,\gamma};$
let $w_{n_\ell,\gamma_\ell}$ $(\ell\in\omega)$
be a countably infinite family in terms of which they can
be expressed.
(If only finitely many are needed, choose the rest arbitrarily.)
Let $a'\in A'$ be
defined again as in~\eqref{d.a'=}, but now using this family of pairs.
Then as noted in the paragraph following~\eqref{d.wn*m},
the elements $w_{n_\ell,\gamma_\ell}$ will behave on
$\{a'\}\times 2^\omega\subseteq A$
like the generators $e_\ell$ in Theorem~\ref{T.2^*w}.
Hence restricting $C$ to a crossword
on $\{a'\}\times 2^\omega\subseteq A,$
and regarding $\{a'\}\times 2^\omega$ as a copy of $2^\omega,$
this crossword will, by the proof of that
theorem, have only finitely many distinct rows and columns.

Moreover, rows of the whole crossword $C$ that
represent distinct elements of $W$ must be described by
distinct elements of the free
distributive lattice generated by the $w_{n_\ell,\gamma_\ell},$ hence
their restrictions to their entries indexed
by elements of $\{a'\}\times 2^\omega$
will also be distinct; so if $C$ has infinitely many distinct rows,
it must have infinitely many distinct columns indexed
by members of $\{a'\}\times 2^\omega;$ hence the induced
crossword on that set will have infinitely many distinct columns,
which we have just seen is impossible.
So such a $C$ can have only finitely many distinct rows;
hence by Corollary~\ref{C.fin_many}, $(A,W)$ is a Pratt comonoid.

Since a discrete Pratt comonoid allows all combinations of
rows to appear in a crossword, results such as~\eqref{d.ctbl_on_A*l}
or Corollary~\ref{C.unctbl_C} show that $(A,W)$ is not discrete.
Let us also note that the cardinality of $W$ is the cardinality of a
free distributive lattice with $\aleph_1$ generators, which
is $\aleph_1.$
We have thus proved:

\begin{theorem}\label{T.T1_nondiscrete}
For $A$ and $W$ as described by~\eqref{d.ceg_A}-\eqref{d.W_frm_wn*m},
$(A,W)$ is a non-discrete $\!T_1\!$ Pratt comonoid, with
$\card(A)=2^{\aleph_0}$ and $\card(W)=\aleph_1.$\qed
\end{theorem}

\section{Can a \texorpdfstring{$\!T_1\!$}{T1} Pratt comonoid structure be countably infinite?  First results}\label{S.free}

We saw in \S\ref{S.countable} that a $\!T_1\!$ Pratt comonoid $(A,W)$
such that $A$ is countable must be discrete.
What if we restrict $W$ rather than $A$?
In the example of the preceding section, $W$ was uncountable because
of the index-set $\omega_1$ occurring in the definition of its
generators $w_{n,\gamma};$ and we needed this uncountability
in proving $\!T_1\!$-ness, i.e., condition~\eqref{d.wn*mT1}.
(We used it again in proving Lemma~\ref{L.unctbl}, but we wouldn't
have had to prove that lemma if in place of the $\omega_1$ in our
definitions we had been able to use, say, $\omega,$ since then
$W$ would have been countable.)

So let us pose:

\begin{question}\label{Q.ctbl_W}
Does there exist an infinite $\!T_1\!$ Pratt comonoid
$(A,W)$ such that $W$ is countable?

If not, does there exist a non-discrete infinite $\!T_1\!$ Pratt
comonoid such that $W$ is countably generated (under
the closure operator of forming diagonals of crosswords)?
\end{question}

We have not been able to answer either of these questions.

Our first thought was that if we took an uncountable set
$A$ and a fairly ``random'' countable
$\!T_1\!$ family of subsets, and closed it under finite unions and
intersections, then the ``randomness'' might
prevent the resulting set $W$ from having
crosswords with infinitely many distinct rows and columns,
so by Corollary~\ref{C.fin_many}, $W$ would be
a Pratt comonoid structure on $A.$

But Corollary~\ref{C.disjoint}
already warns us that there will be difficulties:
$W$ must not have infinite pairwise disjoint families.
In trying to find examples,
the authors came up with some intricate ways of constructing
$\!T_1\!$ lattices of subsets
of a set $A$ having no disjoint pairs of nonempty elements.
For instance, suppose we let $A$ be an
antichain of subsets of $\omega$ (a family of elements none of
which contains another), which can be taken to be uncountable, and for
each natural number $n$ let $e_n$ be the set of members of $A$
which contain $n;$ and consider the Pratt comonoid structure $W$
on $A$ generated by these $e_n.$
The condition that $A$ be an antichain guarantees that
the family $\{e_n\}$ is $\!T_1;\!$ and it is not hard to get
examples where the sublattice generated by
$\{e_n\mid n\in\omega\}\cup\{\emptyset,A\}$
is a free distributive lattice with $0$ and $1$ on the $e_n,$
and hence has no nontrivial pairs of disjoint elements.
Yet when we examined examples of this sort, we found repeatedly
that they admitted unexpected crosswords which brought more subsets
into $W,$ and ultimately led to discrete structures.
We shall now show why something like this was inevitable.

Again, we begin with a result on not-necessarily $\!T_1\!$
structures.

\begin{lemma}\label{L.x_i,y_i}
Let $(A,W)$ be a Pratt comonoid, $x_0\geq x_1\geq\dots$ an
$\!\omega\!$-indexed descending chain of elements of $W,$
and $y_0\leq y_1\leq\dots$ an $\!\omega\!$-indexed ascending chain
of elements of $W,$ such that either
$\bigwedge_{n\in\omega} x_n =\emptyset$
or $\bigvee_{n\in\omega} y_n =A.$

Then $\bigvee_{n\in\omega} x_n\wedge y_n\in W.$
\end{lemma}

\begin{proof}
In the case where $\bigwedge_{n\in\omega} x_n =\emptyset,$
note that any $a\in A$ can belong to only finitely many
of the $x_n,$ hence, a fortiori, can belong to only
finitely many of the sets $x_n\wedge y_n.$
Hence the conclusion is an immediate consequence of
Lemma~\ref{L.near_disjoint}.

Under the alternative hypothesis $\bigvee_{n\in\omega} y_n =A,$
we get the dual of the preceding result,
with the roles of the $x_i$ and $y_i$ interchanged;
but this does not quite give the conclusion we want.
Rather, it says that given $x_0\geq x_1\geq\dots$ and
$y_0\leq y_1\leq\dots$ with $\bigvee_{n\in\omega} y_n=A,$ we have
$\bigwedge_{n\in\omega} x_n\vee y_n\in W.$
Now since $\bigvee_{n\in\omega} y_n=A,$ and the $y_n$ form an increasing
chain, every $a\in A$ lies in almost all the sets $x_n\vee y_n.$
If in fact $a\in y_0,$ it lies in all of them;
if not, it lies in all of them if and only if
it lies in all $x_n$ {\em before} the point where it appears
in $y_n.$
These observations together show that $\bigwedge_{n\in\omega}
x_n\vee y_n=y_0\vee\bigvee_{n\in\omega} x_n\wedge y_{n+1};$
so our dualized result says that
$y_0\vee\bigvee_{n\in\omega} x_n\wedge y_{n+1}\in W.$

If we now apply this result with the sequence $y_n$ replaced
by the sequence that has $\emptyset$ in place of $y_0,$ and
$y_{n-1}$ in place of $y_n$ for $n>0,$ we get the desired conclusion.
\end{proof}

The next result is somewhat more complicated to state, so for
simplicity we will only formulate the conclusion under one of the
two dual hypotheses.
But we remark that if both conditions
$\bigwedge_{m\in\omega} x_m =\emptyset$ and
$\bigvee_{n\in\omega} y_n =A$ hold, then the final hypothesis
says that although the sequence of $\!x\!$'s gets small, and the
sequence of $\!y\!$'s gets large, no $y_n$ contains any $x_m.$
The form shown is the modification of that condition needed
when $\bigvee_{n\in\omega} y_n$ is not necessarily everything.

When we refer to $W$ containing a complete
sublattice, we mean a subset closed in $2^A$ under
(possibly infinite) unions and intersections,
but not necessarily containing $\emptyset$ or $A.$

\begin{proposition}\label{P.continuum}
Again let $(A,W)$ be a Pratt comonoid, $x_0\geq x_1\geq\dots$ an
$\!\omega\!$-indexed descending chain of elements of $W,$
and $y_0\leq y_1\leq\dots$ an $\!\omega\!$-indexed ascending chain
of elements of $W,$ and
suppose that $\bigwedge_{m\in\omega} x_m =\emptyset.$
Suppose moreover that for no pair $(m,n)\in\omega\times\omega$
does $y_n$ contain $x_m\wedge\bigvee_{i\in\omega} y_i.$

Then $W$ has at least continuum cardinality;
in fact, it contains a complete sublattice isomorphic to the
lattice of all subsets of $\omega.$
\end{proposition}

\begin{proof}
Let us show, first, that we can construct sequences of integers
$m(0)<m(1)<\dots$ and $n(0)<n(1)<\dots$  such that for all $i,$
\begin{equation}\begin{minipage}[c]{35pc}\label{d.mi_ni}
$x_{m(i)}\wedge y_{n(i+1)}\ \not\leq
\ (x_{m(i)}\wedge y_{n(i)})\vee(x_{m(i+1)}\wedge y_{n(i+1)}).$
\end{minipage}\end{equation}
We take $m(0)$ and $n(0)$ arbitrary.
Suppose, recursively, that we have found
$m(0),\dots,m(j)$ and $n(0),\dots,n(j)$
so that~\eqref{d.mi_ni} holds for all $i<j.$
By the final hypothesis of the proposition,
$y_{n(j)}\not\geq x_{m(j)}\wedge\bigvee_{i\in\omega} y_i,$
so there is some $a$ in the latter set that is not in the former;
i.e., which lies in $x_{m(j)}$ and some $y_i$ but not in $y_{n(j)}.$
In particular, we can find $n(j{+}1) > n(j)$
such that $a\in y_{n(j+1)};$ moreover, since
$\bigwedge_{m\in\omega} x_m =\emptyset,$ we can find $m(j{+}1) > m(j)$
such that $a\notin x_{m(j+1)}.$
With $a,$ $n(j{+}1),$ and
$m(j{+}1)$ so chosen we see that~\eqref{d.mi_ni} also holds for $i=j.$
Thus we get sequences $m(i)$ and $n(i)$
satisfying~\eqref{d.mi_ni} for all $i\in\omega.$

Now let
\begin{equation}\begin{minipage}[c]{35pc}\label{d.z=}
$z\ =\ \bigvee_{i\in\omega}\,x_{m(i)} \wedge y_{n(i)}.$
\end{minipage}\end{equation}
Since $\{m(i)\}$ is a cofinal subsequence of $\omega,$ we
have $\bigwedge_{i\in\omega} x_{m(i)} = \emptyset,$ hence
by Lemma~\ref{L.x_i,y_i}, $z\in W.$
We claim that if for all $i\in\omega$ we define
\begin{equation}\begin{minipage}[c]{35pc}\label{d.zi=}
$z_i\ =\ z\vee (x_{m(i)}\wedge y_{n(i+1)})$
\end{minipage}\end{equation}
(again an element of $W),$ then
\begin{equation}\begin{minipage}[c]{35pc}\label{d.zi>}
for all $i,$ $z_i>z,$
\end{minipage}\end{equation}
but
\begin{equation}\begin{minipage}[c]{35pc}\label{d.zizj}
for $i\neq j,$ $z_i\wedge z_j = z.$
\end{minipage}\end{equation}
To see~\eqref{d.zi>}, apply~\eqref{d.mi_ni} to get an element
$a\in x_{m(i)}\wedge y_{n(i+1)}$ that is not in
$x_{m(i)}\wedge y_{n(i)}$ or $x_{m(i+1)}\wedge y_{n(i+1)}.$
Thus, $a\in (x_{m(i)} - x_{m(i+1)}) \wedge (y_{n(i+1)} - y_{n(i)}).$
The fact that $a$ is in the first of these two difference-sets
tells us that for $j>i, $ $a\notin x_{m(j)},$
while the fact that it is in
the second set tells us that for $j\leq i,$ $a\notin y_{n(j)},$ so
in each case, $a\notin x_{m(j)}\wedge y_{n(j)}.$
Hence $a\in z_i$ does not lie
in any of the joinands of~\eqref{d.z=}, giving~\eqref{d.zi>}.

To get~\eqref{d.zizj}, note that ``$\geq$'' holds
by~\eqref{d.zi>}, so it suffices to prove ``$\leq$''.
Assume without loss of generality that $i < j.$
Then
\begin{equation}\begin{minipage}[c]{28pc}\label{d.=>zizj}
\hspace*{-3.25pc}$z_i\ \wedge z_j\ =
\ z\vee((x_{m(i)}\wedge y_{n(i+1)})\wedge(x_{m(j)}\wedge y_{n(j+1)}))$
\quad (by~\eqref{d.zi=} and distributivity)\\
$=\ z \vee (x_{m(j)} \wedge y_{n(i+1)})$\hspace{1.1em}
(because $x_{m(j)}\leq x_{m(i)}$ and $y_{n(i+1)}\leq y_{n(j+1)})\\
\leq\ z \vee (x_{m(j)} \wedge y_{n(j)})$\hspace{2em}
(because $i{+}1\leq j,$ so $y_{n(i+1)}\leq y_{n(j)})$\\
$=\ z$ \hspace{9.05em}
(by~\eqref{d.z=}, in particular, the term indexed by $j).$
\end{minipage}\end{equation}

Statements~\eqref{d.zi>} and~\eqref{d.zizj} together say that the $z_i$
are sets properly containing $z,$ which on removing $z$ give
pairwise disjoint sets.
By Lemma~\ref{L.new_fr_old}(iv), the results of removing $z$
from all elements of $W$ which contain it form a
Pratt comonoid on $A-z;$ so the $z_i$ yield pairwise disjoint nonempty
elements of that Pratt comonoid, hence by
Corollary~\ref{C.disjoint}, the
unions of arbitrary subsets of these sets are members of that
Pratt comonoid.
The corresponding elements of $W$ give a complete
lattice of subsets of the asserted form.
\end{proof}

A problem with finding applications of
the above proposition is that it is often hard to
find natural hypotheses on a Pratt comonoid that lead to
a countable descending chain with empty intersection,
or, if we have such a chain, that lead to an ascending chain
which does not eventually ``swallow up'' the descending chain.
However in the next result we get both of these, using
the $\!T_1\!$ assumption together with the condition that $(A,W)$
be generated by the sort of family discussed
following Question~\ref{Q.ctbl_W} above.

When we speak of $(A,W)$ being ``generated by'' a family
$S$ of subsets of $A,$ we mean, of course, that $W$ is the least
Pratt comonoid structure on $A$ which contains $S.$

\begin{theorem}\label{T.continuum}
Suppose $(A,W)$ is a $\!T_1\!$ Pratt comonoid which can
be generated by a countably infinite family $S$ of subsets of
$A$ that satisfy no nontrivial distributive lattice relations
\textup{(}i.e., which form a set of free generators of a free
distributive lattice\textup{)}.
Then $W$ has at least continuum cardinality, and in fact contains a
complete sublattice isomorphic to the lattice of subsets of $\omega.$
\end{theorem}

\begin{proof}
If $A$ is countable, Theorem~\ref{T.countable} gives the desired
conclusion, so assume $A$ uncountable.
Now if we associate to
each $a\in A$ the set of members of $S$ which contain it, then to
distinct elements of $A$ we will associate distinct subsets of $S$
(because $W$ is $\!T_1,\!$ so $S$ must also be).
Hence only countably many
elements of $A$ can yield cofinite subsets of $S,$ so
let $a\in A$ be an element such that
the associated set is non-cofinite in $S.$
Let us now write $S$ as a disjoint union of two infinite sets: a set
$S_1$ which {\em properly} contains the set of elements of $S$
containing $a,$ but is still
non-cofinite; and its complement, $S_2=S-S_1.$

Since $S_1$ properly contains the
set of elements of $S$ containing $a,$ it includes
some element not containing $a,$ so the intersection of the
members of $S_1$ does not contain $a.$
On the other hand, for any $b\neq a,$ the $\!T_1\!$ condition tells us
that some $s\in S$ which contains $a$ (and hence belongs
to $S_1)$ does not contain $b;$ so the intersection of the members
of $S_1$ is empty.

Let us now index each of $S_1$ and $S_2$ by $\omega,$ and for each
$i\in\omega,$ let $x_i$ be the intersection of the first $i$ elements
of $S_1,$ and $y_i$ the union of the first $i$ elements of $S_2.$
We claim that for all $m$ and $n$ we have
$y_n \not\geq x_m\wedge (\bigvee_{i\in\omega} y_i).$
For if $y_n \geq x_m\wedge (\bigvee_{i\in\omega} y_i),$
then in particular, $y_n \geq x_m\wedge y_{n+1},$
which is a nontrivial lattice relation
among the first $m$ elements of $S_1$ and the first $n{+}1$ elements of
$S_2,$ contradicting the assumption on $S.$
So no such relation holds, hence
Proposition~\ref{P.continuum} gives the conclusion of the theorem.
\end{proof}

Did the above proof use the full strength of our hypothesis
that the elements of $S$ satisfy no nontrivial distributive
lattice relations?
If we write $S_1$ as $\{s_{1,i}\mid i\in\omega\}$
and $S_2$ as $\{s_{2,i}\mid i\in\omega\},$
then we find that the condition used at the end of the proof
reduces to the statement $s_{2,0}\vee\dots\vee s_{2,n-1}\not\geq
s_{1,0}\wedge\dots\wedge s_{1,m-1}\wedge s_{2,n}.$
Now the condition that the elements of $S$
satisfy no nontrivial distributive lattice relations
is indeed equivalent to saying that
for no two disjoint finite families of elements of $S$ does the join
of one majorize the meet of the other.
(To see this, note that any lattice word in a set of
elements of a distributive lattice can be reduced to a join
of meets, and also, dually, to a meet of joins.
Hence every relation on a family of elements of such a lattice is
equivalent to the statement that some meet of joins majorize
some join of meets.
But a meet majorizes an element if and only if each meetand
majorizes it; and a join is majorized by an element if and
only if each joinand is majorized by it; so every such relation
reduces to a family of relations each of which says that a join
of elements of our given set majorizes a meet of elements of that set;
hence freeness says that no such nontrivial relation,
i.e., no such relations in which no generator appears
as both a meetand and a joinand, holds.)

However, looking at the quantifications involved in the
proof of the theorem, one finds that one can slightly
restrict the set of relations that one needs to assume do not hold.

\begin{corollary}[to the proof of Theorem~\ref{T.continuum}]\label{C.continuum}
In the situation of Theorem~\ref{T.continuum}, the condition
that $S$ generate a free distributive lattice can be weakened
to say that there exists a partition of $S$ into finitely
many disjoint subsets, $S^{(0)},\dots,\,S^{(k-1)}$ such that no
relation $t_0\vee\dots\vee t_{n-1}\geq s_0\wedge\dots\wedge s_{m-1}$
holding in $W$ with every $t_j$ distinct from
every $s_i$ has all but at most
one of the $s_i$ belonging to the same set $S^{(i)}.$
\end{corollary}

\begin{proof}
Assume $S$ has the above weakened property.
In the proof of Theorem~\ref{T.continuum}, after choosing
$S_1$ to properly contain the set of elements of $S$ containing $a,$
and to have infinite complement, note that this complement
must contain infinitely many members of one of our sets,
say $S^{(i)}.$
So rather than taking $S_2$ to be the full complement of $S_1$
in $S,$ let us take it to consist of the members of that complement
which belong to $S^{(i)}.$
We can now complete the proof as before.
At the last step, if we have a relation
$s_{2,0}\vee\dots\vee s_{2,n-1}\geq
s_{1,0}\wedge\dots\wedge s_{1,m-1}\wedge s_{2,n},$
we note that
the meetands on the right satisfy $s_{1,0},\dots,s_{1,m-1}\in S^{(i)},$
with at most the last meetand $s_{2,n}$ not in $S^{(i)},$
contradicting the hypothesis on $S.$
\end{proof}

We do not know how one might make use of this weaker hypothesis in
studying Question~\ref{Q.ctbl_W}.

\section{Further results on Pratt comonoids with \texorpdfstring{$W$}{W} smaller than the continuum}\label{S.filters}

The arguments in the preceding section used
the fact that if $(A,W)$ is a Pratt comonoid such that
$W$ has cardinality less than the continuum, then $W$ cannot contain
infinitely many pairwise disjoint sets
(by Corollary~\ref{C.disjoint}).
This fact restricts the structure of such comonoids in other
ways as well.

The next lemma seems likely to be known,
but since we do not know a reference, we will give the proof.
We will be applying it to a lattice, the $W$ of a Pratt comonoid
$(A,W),$ but since it only involves the operation $\wedge,$ the
natural context for stating it is that of a $\!\wedge\!$-semilattice.
Recall that this means a
set $L$ with a single idempotent commutative associative
operation $\wedge;$ one then
regards $L$ as partially ordered by taking $x\leq y$ if $x\wedge y=x.$

Let us make:

\begin{definition}\label{D.str_indec}
If $L$ is a $\!\wedge\!$-semilattice with least element $0,$ we
shall call two elements of $L$ {\em disjoint} if their meet is $0,$ and
we will call a nonzero element $x\in L$ {\em strongly indecomposable}
if there do not exist two disjoint nonzero elements $<x$ in $L.$

We will call two strongly indecomposable elements $x,\,y\in L$
{\em equivalent} if they are not disjoint.
\end{definition}

That the above condition is indeed
an equivalence relation on strongly indecomposable elements
is immediate from the definitions.

Three examples:
In the lattice of open-closed subsets of the Cantor set, regarded as a
$\!\wedge\!$-semilattice, there are no strongly indecomposable
elements.
In the lattice whose elements are all the neighborhoods of a point $p$
of a topological space, together with the empty set, which plays
the role of $0,$
all nonzero elements are strongly indecomposable, and form
a single equivalence class.
In the lattice of all subsets of a set $X,$ the singletons
are the strongly indecomposable elements, and no two distinct
singletons are equivalent.

\begin{lemma}\label{L.str_indec}
If $L$ is a $\!\wedge\!$-semilattice with least element $0$
which has no infinite family of pairwise disjoint nonzero elements,
then $L$ has only finitely many equivalence classes of
strongly indecomposable elements, and every nonzero element of $L$
majorizes at least one strongly indecomposable element.
\end{lemma}

\begin{proof}
There cannot be infinitely many equivalence classes of strongly
indecomposable elements of $L,$ because
a family of representatives of such equivalence classes would
form an infinite family of pairwise disjoint elements.

To prove that every nonzero $x\in L$ majorizes
at least one strongly indecomposable element, suppose $x>0$ does not.
In particular, $x$ is not itself strongly indecomposable, so there
exist disjoint nonzero elements $x_0,\,y_0<x.$
By our assumption on $x,$ the
element $x_0$ is not strongly indecomposable, so
it in turn majorizes disjoint elements $x_1,\,y_1<x_0.$
Continuing in this manner, we see that $y_0,$ $y_1,\dots$
will form an infinite family of pairwise disjoint nonzero elements,
contradicting the hypothesis on $L.$
\end{proof}

In the situation above, if $E_0,\dots,E_{n-1}$ are the
equivalence classes of strongly indecomposable elements, and
we map each $x\in L$ to the set of those $E_i$ such
that $x$ majorizes a member of $E_i,$ this yields
a homomorphism of $\!\wedge\!$-semilattices from $L$
to the $\!\wedge\!$-semilattice of subsets of
$\{E_0,\dots,E_{n-1}\},$ which sends only $0$ to the empty set.

By Corollary~\ref{C.disjoint}, if $(A,W)$ is a Pratt comonoid such
that $W$ has less than continuum cardinality,
the above lemma is applicable to $W.$
More generally, for every $w$ in such a $W,$ that lemma is
applicable to the lattice of members of $W$ containing $w$
by Lemma~\ref{L.new_fr_old}(iv); moreover these same statements
apply to the dual lattice to $W$ by Lemma~\ref{L.new_fr_old}(i).

Given elements $w<x$ in a $\!\wedge\!$-semilattice, let
us call $x$ strongly indecomposable {\em relative to} $w$ if
it is strongly indecomposable in the $\!\wedge\!$-semilattice
$\{y\in W\mid y\geq w\}$ (with $w$ regarded as $\!0\!$-element).
Likewise, in a $\!\vee\!$-semilattice $L$ with greatest element
$1,$ let us call an element $x$
{\em dually strongly indecomposable} if it is strongly indecomposable
in the dual $\!\wedge\!$-semilattice; and define dual strong
indecomposability relative to an element $w>x$ in the obvious way.
Then the above observations give:

\begin{corollary}\label{C.str_indec}
Suppose $(A,W)$ is a Pratt comonoid such that $W$ has less
than continuum cardinality.
Then for every $w\in W$ \textup{(}including $\emptyset),$
the set of equivalence classes of elements
strongly indecomposable relative to $w$ is finite, and every
element $>w$ majorizes at least one such element.
Likewise, the set of equivalence classes of elements
{\em dually} strongly indecomposable relative to $w$ is finite, and
every element $<w$ is majorized by at least one such element.\qed
\end{corollary}

Not only elements of $W,$ but also elements of $A$ have
obligatory relationships with equivalence classes of
strongly indecomposable elements:

\begin{lemma}\label{L.a,E}
Suppose $(A,W)$ is a Pratt comonoid such that $W$ has less
than continuum cardinality.
Let us say that an element $a\in A$ {\em dominates} an
equivalence class $E$ of strongly indecomposable elements
if every member of $W$ which contains the element $a$ majorizes some
member of $E.$

Then every $a\in A$ dominates one or more
equivalence classes of strongly indecomposable elements of $W.$

Moreover, if $(A,W)$ is $\!T_1\!$ and $a$ dominates an equivalence
class $E,$ then either $\bigwedge_{w\in E} w=\emptyset,$
or $\bigwedge_{w\in E} w=\{a\}.$
In the latter case, $E$ is the set of all strongly
indecomposable elements of $W$ containing $a,$
$E$ is the only equivalence class dominated by $a,$
and $a$ is the only element dominating $E.$

Hence if $(A,W)$ as above is $\!T_1\!$ and $A$ is infinite, there
exists at least one
equivalence class $E$ of strongly indecomposable elements
such that $\bigwedge_{w\in E} w=\emptyset.$
\end{lemma}

\begin{proof}
Suppose some $a\in A$ dominated no equivalence class $E$ of strongly
indecomposable elements.
Then for each such $E$ we could find
a $w_E\in W$ containing $a$ but majorizing no member of $E.$
The intersection of these finitely many elements would be a member
of $W$ containing $a,$ hence nonempty, but majorizing no
strongly indecomposable element, contradicting Lemma~\ref{L.str_indec}.
This gives our first conclusion.

If $(A,W)$ is $\!T_1\!$ and $a$ dominates $E,$
then for every $b\neq a$ we can find a member of $W$
containing $a$ but not $b,$ and this will majorize an element of $E;$
hence $\bigwedge_{w\in E} w$ can contain no $b\neq a,$
and so must be $\emptyset$ or $\{a\}.$
In the latter case, it is easy to verify
that the set of strongly indecomposable elements of $W$
containing $a$ must coincide with $E,$
giving the first two assertions under this hypothesis.
Moreover, for any $b\neq a,$ the $\!T_1\!$ property gives a member
of $W$ containing $b$ but not $a,$ giving the third assertion.

Since there are only finitely many $E,$ there are only finitely
many $a$ dominating equivalence classes $E$ with nonempty
intersection; so if $A$ is infinite, we can apply the
first conclusion of the lemma to any $a$ not of that sort,
and get the final conclusion.
\end{proof}

Further observations:
If $(A,W)$ is a Pratt comonoid with $W$ countably infinite,
and $E$ an equivalence class of strongly indecomposable
elements of $W$ having empty intersection, then we see that
we can construct an $\!\omega\!$-indexed descending chain
$x_0>x_1>\dots$ downward cofinal in $E,$
and so in particular, having empty intersection.
Given the countability of $W,$ Proposition~\ref{P.continuum}
implies that for any ascending chain $y_0<y_1<\dots$ in $W,$
there must be some $m$ and $n$ such that
\begin{equation}\begin{minipage}[c]{35pc}\label{d.geq}
$y_n\ >\ x_m\wedge\bigvee_{i\in\omega} y_i.$
\end{minipage}\end{equation}
But in fact, we had enough information to see that
without calling on Proposition~\ref{P.continuum}.
On the one hand,~\eqref{d.geq} holds for all $m$ and $n$
if $x_0\wedge\bigvee_{i\in\omega} y_i=\emptyset.$
On the other hand, if this intersection is nonempty, that means
some $y_n$ has nonempty intersection with $x_0.$
That intersection belongs to $E,$
hence will majorize some $x_m,$ and the relation $y_n>x_m$
then implies~\eqref{d.geq}.

Though such increasing and decreasing chains do not have the
properties that would allow us to call on
Proposition~\ref{P.continuum} and get a contradiction, they
do lead to the following result, which
contrasts with the behavior of the examples of
Theorems~\ref{T.2^*w} and~\ref{T.T1_nondiscrete}.

\begin{proposition}\label{P.inf_crswd}
If $(A,W)$ is a $\!T_1\!$ Pratt comonoid with $W$ countably infinite,
then there exist crosswords over $W$ with infinitely many
distinct rows and columns.
\end{proposition}

\begin{proof}
As noted above, we can construct an infinite descending
chain $(x_m)$ with empty intersection, and by duality,
an infinite ascending chain $(y_n)$ with union $A.$
Without loss of generality, let us take these chains
to be strictly decreasing and strictly increasing respectively.
Then $C=\bigvee_n x_n\times y_n$ will be the desired crossword.
Indeed, given $a\in A,$ let $m$ be the greatest integer
such that $a\in x_m;$ then we see that the row of $C$ indexed
by $a$ will be given by $y_m\in W,$ and a dual statement applies
to columns; so the rows and columns of $C$ comprise precisely
$\{x_m\}\cup\{y_n\},$ an infinite subset of $W.$
\end{proof}

Here is another sort of chain that we can
construct in any infinite $\!T_1\!$ Pratt comonoid with $W$ countable,
or more generally, with $\card(W)$ both less than the
continuum and less than $\card(A).$
The relativized version of Lemma~\ref{L.a,E} shows that
for every $w\in W$ there are
only finitely many elements $a\notin w$ which
dominate equivalence classes relative to $w$ that have
intersection strictly larger than $w.$
Since $\card(W)<\card(A),$ there must be an $a$ which, relative to
every $w$ not containing it, dominates no such equivalence classes.
Now starting a recursion with $y_0=\emptyset,$ assume we have
elements $y_0<\dots<y_{i-1},$ each strongly indecomposable
relative to the one before, with $a\notin y_{i-1}.$
Then we can take an equivalence class $E_i$
of elements strongly indecomposable relative to $y_{i-1}$ which is
dominated by $a$ relative to $y_{i-1},$ and take a $y_i\in E_i$ which
does not contain~$a.$
We thus get an infinite chain with these properties;
but how this fact might be useful we again don't know.

If we assume the negation of the continuum hypothesis,
then the example of Theorem~\ref{T.T1_nondiscrete}
has $W$ of less than continuum cardinality, hence
Corollary~\ref{C.str_indec}, Lemma~\ref{L.a,E}, and
the observation of the last paragraph apply to that example.
Thus, we cannot expect those
results to yield contradictions without some
stronger assumption, such as that $W$ be countable.

A question we have not studied, which is also suggested
by Theorem~\ref{T.T1_nondiscrete}, is:

\begin{question}\label{Q.A<2^*2}
Assuming the negation of the continuum hypothesis, can
there exist a non-discrete $\!T_1\!$ Pratt comonoid
$(A,W)$ whose base-set $A$ has less than continuum cardinality?
\end{question}

\section{Generalizing the construction of \texorpdfstring{\S\ref{S.2^*w}}{section 5}}\label{S.*P_Pi}

Let us pick up a loose end.
In \S\ref{S.2^*w} we saw that the lattice of subsets of $A=2^\omega$
generated by $\{e_n\mid n\in\omega\}\cup\{\emptyset,A\}$
formed a Pratt comonoid.
We sketch here, as promised in the paragraph preceding
Corollary~\ref{C.isolated}, a generalization of that result.

First let us generalize Corollary~\ref{C.isolated}.
For brevity we will not, this time, include condition~(ii)
of Lemma~\ref{L.isolated} among the equivalent conditions
in the statement, though the equivalence of conditions~(i) and~(ii)
of that lemma will still be essential to the proof.

\begin{lemma}\label{L.prod_A_i}
Let $(A_i)_{i\in I}$ be any family of {\em finite} partially
ordered sets, such that each $A_i$ has a least element $0_i$
and a greatest element $1_i,$ and let $A=\prod_{i\in I} A_i,$
with $\preccurlyeq$ the componentwise partial ordering.
Let $S$ be the set of elements of $A$ which have $\!i\!$-th
coordinate $0_i$ for all but at most one $i,$
and $S'$ the set of elements having $\!i\!$-th
coordinate $1_i$ for all but at most one $i.$
Then the following conditions on an element $x\in U_\subseteq(A)$
are equivalent.

\textup{(i)} $x$ is an isolated point of $U_\subseteq(A)$
under the natural topology.

\textup{(ii)} $x$ is the inverse image under the
projection of $A$ onto a finite sub-product
$A_{i_0}\times\dots\times A_{i_{n-1}}$
\textup{(}where $i_0,\dots,i_{n-1}$ are distinct elements
of $I)$ of an up-set of that sub-product.

\textup{(iii)} $x$ lies in the lattice of up-sets of $A$ generated
by $\emptyset$ and the elements ${\uparrow}(s)$ for $s\in S.$

\textup{(iv)} $\neg x$ lies in the lattice of down-sets of $A$ generated
by $\emptyset$ and the elements ${\downarrow}(s)$ for $s\in S'.$
\end{lemma}

\begin{proof}
We shall show (i)$\!\implies\!$(ii)$\!\iff\!$(iii), whence
by symmetry also (ii)$\!\iff\!$(iv),
and then that (iii)$\!\wedge\!$(iv)$\!\implies\!$(i).

The proof that (i)$\!\implies\!$(ii) follows the idea of
Corollary~\ref{C.isolated}~(i)$\wedge$(ii)$\implies$(iii).
Assuming~(i), we can, by Lemma~\ref{L.isolated}(ii),
write $x$ as a union ${\uparrow}(a_0)\vee\dots\vee{\uparrow}(a_{m-1})$
with each $a_j\in A,$ and we can assume without loss of
generality that none of the $a_j$ majorizes any of the others.
If any of the $a_j$ had infinitely many nonzero coordinates,
then we could write ${\uparrow}(a_j)$ as an intersection
of principal up-sets determined by elements ${\uparrow}(a)$ which
take finitely many nonzero coordinates from $a_j,$ and have
zero-elements in all other coordinates.
As in the proof of that corollary, this would make
$x$ a limit, in the natural topology on $U_\subseteq(A),$
of elements $\neq x,$ contradicting~(i).
Hence $x$ is a join of finitely many elements each
constraining only finitely many coordinates of elements of $A.$
If we write $\{i_0,\dots,i_{n-1}\}$ for the full set of
indices whose coordinates $x$ constrains, we see that $x$
is the inverse image of a subset of
$A_{i_0}\times\dots\times A_{i_{n-1}},$ which will
be an up-set in that product, proving~(ii).

To get (ii)$\!\implies\!$(iii), note that since
$A_{i_0}\times\dots\times A_{i_{n-1}}$ is a finite
partially ordered set, every up-set of that set is
a finite (possibly empty) union of principal up-sets.
Writing such a principal up-set
as ${\uparrow}(a_0,\dots,a_{n-1}),$ we see that it is the
intersection over $j=0,\dots,n-1$ of the principal up-sets determined
by the elements which have $a_j$ in the $i_j$ coordinate
and zeroes in all other coordinates.
Hence the inverse image of that up-set in $A$ is the intersection of
the principal up-sets of $A$ having the same descriptions, but
with ``other coordinates'' now ranging over $I$ rather
than just $\{i_0,\dots,i_{n-1}\}.$
This leads to the description of $x$ as in~(iii).
The reverse implication is clear.

By symmetry, we likewise get (ii)$\!\iff\!$(iv).

Finally, (iii)$\!\wedge\!$(iv) immediately gives
Lemma~\ref{L.isolated}(ii), and hence~(i).
\end{proof}

This leads to the following generalization of Theorem~\ref{T.2^*w}.

\begin{theorem}\label{T.prod_A_i}
For $A$ a partially ordered set
constructed as in Lemma~\ref{L.prod_A_i}, the set $W$ of all
up-sets of $A$ that satisfy the equivalent conditions
of that lemma is a Pratt comonoid structure on $A.$
\end{theorem}

\begin{proof}[Sketch of proof]
Let us note how to adapt the proof of Theorem~\ref{T.2^*w}.
Condition~(iii) of Lemma~\ref{L.prod_A_i}, which
is analogous to the hypothesis of that theorem,
makes for the easiest translation of the proof.
We replace occurrences of ``$2$'' in that proof,
where they represent value-sets
of coordinates of elements of $A,$ by the appropriate
finite partially ordered sets $A_i,$ while
where $2$ occurs as the value-set of members
of $W\subseteq 2^A,$ it remains unchanged.

In Theorem~\ref{T.2^*w}, the generators $e_n: A\to 2$ of $W$ were
projections to the $\!n\!$-th coordinates.
The corresponding generators ${\uparrow}(s)$ $(s\in S)$ of
our present $W$ can be regarded as composite maps
$A\to A_i\to 2,$ where the first arrow is the projection onto
the $\!i\!$-th component, and the second is the characteristic
function of the principal up-set determined by some element of $A_i.$
Thus, as in the proof of Theorem~\ref{T.2^*w}, they are continuous maps.
Where $A$ was included in our list of lattice generators in
Theorem~\ref{T.2^*w}, it does not require separate mention here, since
it can be written ${\uparrow}(0),$ but $\emptyset$ is still needed,
and indeed appears in Lemma~\ref{L.prod_A_i}(iii).
\end{proof}

\section{Appendix: Background on the concept of Pratt comonoid}\label{S.chu}

In \cite{chu}, \cite{chu_online}, \cite{comonoid}
Vaughan Pratt studies, for $\Sigma$ a set, the category
$\mathbf{chu}_\Sigma,$ whose objects, {\em Chu spaces},
are pairs $(A,r,X),$
where $A$ and $X$ are sets, and $r:A\times X\to \Sigma$ a set map,
and where a morphism $(A,r,X)\to (B,s,Y)$ is given by
a pair of set-maps, $f:A\to B$ and $g:Y\to X$ such that
$s(f(a),y)=r(a,g(y))$ for $a\in A,$ $y\in Y.$
These spaces are used to model various programming concepts.
It is noted in \cite[\S1.6]{chu_online} that the definition is
based on ideas from the Master's thesis of Po Hsiang Chu.

If one restricts attention to objects $(A,r,X)$
such that distinct elements of $X$ induce distinct maps on $A,$
then one can regard $X$ as a set of maps $x: A\to\Sigma,$ and
drop the map $r$ from the description of these objects.
If one also takes $\Sigma=2=\{0,1\},$ then
$X,$ now a set of $\!\{0,1\}\!$-valued
functions on $A,$ can be regarded as a
set of distinguished subsets of $A,$
and morphisms $(A,X)\to (B,Y)$ correspond to set-maps $A\to B$
under which the inverse image of every distinguished
subset of $B$ is a distinguished subset of $A.$
Pratt notes that various sorts of mathematical structures
can be described as instances of Chu spaces;
for instance, the category of topological spaces can
be considered a subcategory of $\mathbf{chu}_2,$ determined by
the condition that the distinguished subsets
are closed under arbitrary unions and finite intersections.

Given objects $(A,X)$ and $(B,Y)$ of $\mathbf{chu}_\Sigma,$
Pratt defines $(A,X)\otimes(B,Y)$ to be the object whose
first component is the set $A\times B,$ and whose second component
is the set of those maps $A\times B\to\Sigma$
which form ``crosswords'' with rows from $Y$ and columns from $X.$
He then defines a {\em comonoid} in $\mathbf{chu}_\Sigma$ to
be an object $(A,X)$ given with a map
$(A,X)\to (A,X)\otimes (A,X)$ which makes certain diagrams
commute, dual to the diagrams of set-maps
that define the ordinary concept of monoid.

(The first author of this note, having worked with coalgebra objects
as representing objects for algebra-valued functors \cite{coalg},
prefers to use the unmodified term ``comonoid''
for an object given with an appropriate
sort of morphism into the {\em coproduct} of two copies of itself,
and would call an object of the sort Pratt considers
a ``$\!\otimes\!$-comonoid''.)

Pratt then shows that a morphism
$(A,X)\to (A,X)\otimes (A,X)$ can satisfy this
definition of comonoid if and only if it is
determined by the diagonal map $A\to A\times A;$
so such comonoids correspond to
objects $(A,X)$ with the property that
their diagonal maps {\em are} morphisms; in other
words, that every ``crossword'' over $X$
determines, via its diagonal, an element of $X.$

In \cite{comonoid} and \cite{puzzle}, Pratt focuses
on the case $\Sigma=2,$ and, in view of the fact
that these comonoids can be developed, for the
nonspecialist, in language that does not require familiarity
with the category $\mathbf{chu}_2,$
defines them roughly as we have done here,
emphasizing the ``crossword'' metaphor.
We have deviated from his notation only in
replacing $X$ with $W,$ as a mnemonic for ``words''.


\begin{thebibliography}{00}

\bibitem{coalg} George M.~Bergman and Adam O.~Hausknecht,
{\em Cogroups and Co-rings in Categories of Associative Rings,}
A.\,M.\,S. Mathematical Surveys and Monographs,
v.45, ix$\!+\!$388\,pp., 1996.
MR1387111

\bibitem{abc} M.\,Ern\'{e},
{\em The ABC of order and topology,} pp.\,57--83 in
{\em Category Theory at Work \textup{(}Bremen, 1990\textup{)}}
Res. Exp. Math., 18, Heldermann, Berlin, 1991.
MR1147919

\bibitem{ME} Marcel Ern\'{e},
{\em Infinite distributive laws versus local connectedness and
compactness properties,}
Topology Appl.\ {\bf 156} (2009) 2054--2069.
MR2532134

\bibitem{GHKLMS1980} Gerhard Gierz, Karl Heinrich Hofmann,
Klaus Keimel, Jimmie D. Lawson, Michael W. Mislove, Dana S. Scott,
{\em A Compendium of Continuous Lattices},
Springer-Verlag, Berlin-New York, 1980.
MR0614752

\bibitem{GHKLMS2003} G. Gierz, K. H. Hofmann,
K. Keimel, J. D. Lawson, M. Mislove, D. S. Scott,
{\em Continuous Lattices and Domains,}
Encyclopedia of Mathematics and its Applications, 93.
Cambridge University Press, Cambridge, 2003.
MR1975381

\bibitem{chu} Vaughan Pratt,
{\em Chu spaces,} in {\em School on Category Theory and Applications
\textup{(}Coimbra, 1999\textup{)}, 39--100},
Textos Mat. S\'{e}r. B, 21, Univ. Coimbra, Coimbra, 1999.
MR1774536

\bibitem{chu_online} Vaughan Pratt,
{\em Chu Spaces,
Notes for the School on Category Theory and Applications,
University of Coimbra, July 13-17, 1999}.  70\,pp., 1999.
\url{http://boole.stanford.edu/pub/coimbra.pdf}

\bibitem{comonoid} Vaughan R. Pratt,
{\em Comonoids in chu: a large Cartesian closed sibling of
topological spaces,} 12~pp. (electronic), in
{\em CMCS'03: Coalgebraic Methods in Computer Science,}
Proceedings of the 6th Workshop held in Warsaw, April 5--6, 2003.
Ed.\ H. Peter Gumm. Electronic Notes in Theoretical Computer Science,
{\bf 82} No.~1, Elsevier, 2003.
\url{http://boole.stanford.edu/pub/comonoids.pdf}\,.
MR2194154 (for the volume containing this article)

\bibitem{puzzle} Vaughan Pratt,
{\em Puzzle~1.5} (2003) in {\em A crossword puzzle,}
\url{http://thue.stanford.edu/puzzle.html}\,.

\end{thebibliography}
\end{document}